\documentclass[11pt, a4paper,leqno]{amsart}
\usepackage{amsmath,amsthm,amscd,amssymb,amsfonts, amsbsy}
\usepackage{latexsym}
\usepackage{txfonts}
\usepackage{exscale}

\usepackage[colorlinks,citecolor=red,pagebackref,hypertexnames=false]{hyperref}
\usepackage{color}

\calclayout
\allowdisplaybreaks

\theoremstyle{plain}
\newtheorem{theorem}[equation]{Theorem}
\newtheorem{lemma}[equation]{Lemma}
\newtheorem{corollary}[equation]{Corollary}
\newtheorem{proposition}[equation]{Proposition}

\theoremstyle{definition}
\newtheorem{definition}[equation]{Definition}

\theoremstyle{remark}
\newtheorem{remark}[equation]{Remark}

\newtheorem{claim}[equation]{Claim}

\numberwithin{equation}{section}

\newcommand{\dist}{\operatorname{dist}}

\newcommand{\dv}{\operatorname{div}}

\newcommand{\re}{\mathbb{R}}
\newcommand{\rn}{\mathbb{R}^n}

\newcommand{\ree}{\mathbb{R}^{n+1}}

\newcommand{\dd}{\mathbb{D}}

\newcommand{\C}{\mathcal{C}}

\newcommand{\om}{\Omega}

\newcommand{\oT}{\widetilde{\Omega}}

\newcommand{\F}{\mathcal{F}}
\newcommand{\Qt}{\widetilde{Q}}

\newcommand{\M}{\mathcal{M}}

\newcommand{\W}{\mathcal{W}}
\newcommand{\B}{\mathcal{B}}

\newcommand{\I}{\mathcal{I}}

\newcommand{\sbf}{{\bf S}}

\newcommand{\G}{\mathcal{G}}

\newcommand{\GS}{\Gamma_{\sbf}}

\newcommand{\mut}{\mathfrak{m}}

\newcommand{\pom}{\partial\Omega}

\newcommand{\hm}{\omega}
\newcommand{\vp}{\varphi}

\renewcommand{\emptyset}{\mbox{\textup{\O}}}

\DeclareMathOperator{\diam}{diam}

\DeclareMathOperator{\interior}{int}

\begin{document}
\allowdisplaybreaks

\title[Approximation of Uniformly Rectifiable Sets]{Harmonic Measure and
Approximation of Uniformly Rectifiable Sets}

\author{Simon Bortz}

\address{Simon Bortz
\\
Department of Mathematics
\\
University of Missouri
\\
Columbia, MO 65211, USA} \email{sabh8f@mail.missouri.edu}

\author{Steve Hofmann}

\address{Steve Hofmann
\\
Department of Mathematics
\\
University of Missouri
\\
Columbia, MO 65211, USA} \email{hofmanns@missouri.edu}

\thanks{The authors were supported by NSF grant DMS-1361701.}

\date{\today}


\begin{abstract} 
Let $E\subset \ree$, $n\ge 1$, be a uniformly rectifiable set of dimension $n$.  
We show $E$ that has big pieces of boundaries of a class of domains which satisfy a 
2-sided corkscrew condition, and whose connected components are 
all chord-arc domains (with uniform control of the various constants). 
As a consequence, we deduce
that $E$ has big pieces of sets for which harmonic measure belongs to weak-$A_\infty$.
\end{abstract}

\maketitle

\tableofcontents

\section{Introduction}

The results in this paper grew out of a project to prove higher dimensional, quantitative
versions of the classical F. and M. Riesz Theorem \cite{Rfm}.   The latter states
that for a simply connected domain in the complex plane,
with a rectifiable boundary, harmonic measure is absolutely continuous with respect to arclength 
measure.  A quantitative version of this theorem (again in the plane) was
obtained by Lavrentiev \cite{Lav}.   We note that some connectivity hypothesis is essential to these results:
indeed, Bishop and Jones
\cite{BiJo} have  presented a counter-example
to show that the result of \cite{Rfm}
may fail in the absence of sufficient connectivity.  
Thus, roughly speaking, rectifiability plus connectivity 
implies absolute continuity, but rectifiability alone does not.

In higher dimensions, quantitative (scale-invariant) versions of the F. and M. Riesz Theorem were obtained by
Dahlberg \cite{D1} in Lipschitz domains, and by David and Jerison \cite{DJe}, and independently, by Semmes
\cite{Se}, in NTA domains with ADR (Ahlfors-David Regular) boundaries (all terminology and notation to be 
defined below).    
In these quantitative results, the conclusion is that
harmonic measure satisfies a scale invariant version of absolute continuity 
with respect to surface measure, namely that it belongs to the Muckenhoupt class $A_\infty$.
To draw the analogy with the result of \cite{Rfm} more precisely, we note that recently, in \cite{AHMNT} it
has been shown that for a domain $\Omega$ satisfying a scale invariant connectivity hypothesis (the so called 
``uniform" condition, which is a unilateral version of the NTA property), whose boundary is 
UR (Uniformly Rectifiable, a scale invariant version of rectifiability which entails, in particular, the 
ADR property), then in fact $\Omega$ is NTA, so that the result of \cite{DJe} and \cite{Se} applies.
An earlier, direct proof of the scale invariant absolute continuity of harmonic measure 
with respect to surface measure,  in a uniform domain with a UR boundary, appears in \cite{HM-I}.
The converse is also true, see \cite{HMU}.

As noted above, by the counter-example of \cite{BiJo}, such results cannot hold in the absence of
some connectivity hypothesis.  Nonetheless,
in this paper, we obtain a structure theorem for uniformly rectifiable sets of co-dimension 1,
which yields in particular that the F. and M. Riesz theorem holds for every such set $E$ 
(viewed as the boundary of an open set $\om=\ree\setminus E$), 
in a ``big pieces" sense.   
Our main result (the structure theorem) is the following
(our terminology and notation will be defined in the sequel; in particular, however, we let $\dd(E)$ denote the
collection of ``dyadic cubes" on the set $E$, as per David and Semmes \cite{DS1} and M. Christ \cite{Ch};
see Lemma \ref{lemmaCh} below).

\begin{theorem}\label{t1} Let $E\subset \ree$ be a UR (uniformly rectifiable)  set of dimension $n$.
Then for each $Q \in \mathbb{D}(E)$ there exists an open set $\oT=\oT_Q \subset \om:= \ree\setminus E$, 
with $\diam(\oT)\approx \diam(Q)$, such that $\oT$ has an
ADR (Ahlfors-David Regular) boundary, satisfies a 2-sided corkscrew condition,
and 
\begin{equation}\label{t1eq1}
\sigma(\partial\oT \cap Q) \gtrsim \sigma(Q).
\end{equation}
Moreover, each connected component of $\oT$ is an NTA domain with ADR boundary. The various NTA, ADR, and implicit 
constants are uniformly controlled, and depend only
on dimension and on the UR character of $E$.
\end{theorem}

We remark that, in particular, Theorem \ref{t1} says that $E$ has big 
pieces of sets satisfying a 2-sided corkscrew condition, and thus, by a result of 
David and Jerison \cite{DJe} (see also \cite{DS3}), 
has ``Big Pieces of Big Pieces of Lipschitz Graphs" ($BP^2(LG)$; see Definition \ref{bpdef} below). 
Theorem \ref{t1} therefore yields as an immediate corollary the co-dimension 1 case of a result of
Azzam and Schul \cite{AS}.
\begin{corollary} Let $E\subset \ree$ be a UR (uniformly rectifiable)  set of co-dimension 1. 
Then $E\in BP^2(LG)$.
\end{corollary}
We should note that, in fact, the result of \cite{AS} establishes $BP^2(LG)$ for $UR$ sets in all 
co-dimensions, whereas our arguments do not address the case of co-dimension greater than 1.  On the 
other hand, in the co-dimension 1 case, our Theorem \ref{t1} yields extra structure which allows us to 
obtain estimates for harmonic measure.  More precisely, we have the following.

\begin{theorem}\label{t2} Let $E\subset \ree$ be an $n$-dimensional UR set.
Let $\om := \ree\setminus E$. 
Then $E$ has ``interior big pieces of good harmonic measure estimates" (IBP(GHME))
in the  following sense: 
for each $Q \in \dd (E)$ there exists an open set $\oT=\oT_Q \subset \om$, 
with $\diam(\oT) \approx \diam(Q)$, such that $\oT$ satisfies a strong 2-sided  
corkscrew condition along with estimate (\ref{t1eq1}), and for each surface ball
$\Delta = \Delta(x,r) := B(x,r) \cap \partial\oT$, with $x \in \partial\oT$ and $r \in (0, \diam(\oT))$, 
and with interior corkscrew point $X_\Delta$, it holds that
$\hm^{X_\Delta}:=\hm^{X_\Delta}_{\widetilde{\om}}$, the harmonic measure for $\oT$ with pole at $X_\Delta$, belongs to 
weak-$A_\infty(\Delta)$.
\end{theorem}

Thus, every UR set of co-dimension 1 has big pieces of sets satisfying a quantitative, scale 
invariant F. and M. Riesz Theorem.   We remark that this fact actually characterizes uniformly rectifiable sets of
co-dimension 1, as the second named author will show in a forthcoming joint paper with J. M. Martell \cite{HM-IV}.

For the sake of clarity, we note that in the statements of Theorems \ref{t1} and \ref{t2}, we use the notation
$$\Omega:= \ree\setminus E\,,$$ 
where $E\subset \ree$ is in particular
an $n$-dimensional ADR set (hence closed);
thus $\Omega$ is open, but need not be a connected domain.  The open set 
$\widetilde{\om}\subset \Omega$ in Theorem \ref{t2} is the one that we construct in Theorem \ref{t1}.

We remark that the weak-$A_\infty$ conclusion of Theorem \ref{t2} is in the nature of best possible.
Indeed, fix $X\in \Omega$, let $\hat{x}\in E$ be such that
$|X-\hat{x}|=\dist(X,E)=:\delta(X)$, and consider the ball $B_X:= B(\hat{x},10\delta(X))$,
and corresponding surface ball $\Delta_X:= B_X\cap E$.  Choose $Q\in\dd(E)$ such that
$\diam(Q)\approx \delta(X)$, with $Q\subset \Delta_X$.  Our construction in the proof of Theorem
\ref{t1} will yield that
$X \in \widetilde{\om}_Q$, and in fact is a Corkscrew point  for a surface ball 
$\widetilde{\Delta}\subset\partial \widetilde{\om}_Q$, of radius $r\approx \delta(X)$, which contains
$Q\cap\partial\widetilde{\om}_Q$.  Consequently, if $\hm^X_{\widetilde{\om}}$ were in 
$A_\infty(\widetilde{\Delta})$ (rather than merely weak-$A_\infty$), 
then by the maximum principle, letting $\hm^X$ denote harmonic measure for $\om$,
we would have
\begin{equation}\label{eq1.5*}
A\subset \Delta_X, \,\, H^n(A)\geq (1-\eta)H^n(\Delta_X) \,\,\implies\,\, \hm^X(A)\geq
\hm^X_{\widetilde{\om}}(A\cap \widetilde{\Delta})\geq c>0\,,
\end{equation}
for some uniform positive constant $c$, provided that $\eta\in (0,1)$ was sufficiently small depending only on
the ADR constants for $E$ and for $\partial\widetilde{\om}$, and the
implicit constant in \eqref{t1eq1}.  In turn, by the result of \cite{BL}, \eqref{eq1.5*} for every
$X\in \om$ would then imply that
$\hm^X\in$ weak-$A_\infty(\Delta_X)$, which cannot hold in general, by the example of \cite{BiJo}.

Let us now define the terms used in the statements of our theorems.
Most of the following notions have meaning in co-dimensions greater than 1, but here
we shall discuss only
the co-dimension 1 case that is of interest to us in the present work. 

\begin{definition}\label{defadr} ({\bf  ADR})  (aka {\it Ahlfors-David regular}).
We say that a  set $E \subset \ree$, of Hausdorff dimension $n$, is ADR
if it is closed, and if there is some uniform constant $C$ such that
\begin{equation} \label{eq1.ADR}
\frac1C\, r^n \leq \sigma\big(\Delta(x,r)\big)
\leq C\, r^n,\quad\forall r\in(0,\diam (E)),\ x \in E,
\end{equation}
where $\diam(E)$ may be infinite.
Here, $\Delta(x,r):= E\cap B(x,r)$ is the ``surface ball" of radius $r$,
and $\sigma:= H^n|_E$ 
is the ``surface measure" on $E$, where $H^n$ denotes $n$-dimensional
Hausdorff measure.
\end{definition}

\begin{definition}\label{defur} ({\bf UR}) (aka {\it uniformly rectifiable}).
An $n$-dimensional ADR (hence closed) set $E\subset \ree$
is UR if and only if it contains ``Big Pieces of
Lipschitz Images" of $\rn$ (``BPLI").   This means that there are positive constants $\theta$ and
$M_0$, such that for each
$x\in E$ and each $r\in (0,\diam (E))$, there is a
Lipschitz mapping $\rho= \rho_{x,r}: \rn\to \ree$, with Lipschitz constant
no larger than $M_0$,
such that 
$$
H^n\Big(E\cap B(x,r)\cap  \rho\left(\{z\in\rn:|z|<r\}\right)\Big)\,\geq\,\theta\, r^n\,.
$$
\end{definition}

We recall that $n$-dimensional rectifiable sets are characterized by the
property that they can be
covered, up to a set of
$H^n$ measure 0, by a countable union of Lipschitz images of $\rn$;
we observe that BPLI  is a quantitative version
of this fact.

We remark
that, at least among the class of ADR sets, the UR sets
are precisely those for which all ``sufficiently nice" singular integrals
are $L^2$-bounded  \cite{DS1}.    In fact, for $n$-dimensional ADR sets
in $\ree$, the $L^2$ boundedness of certain special singular integral operators
(the ``Riesz Transforms"), suffices to characterize uniform rectifiability (see \cite{MMV} for the case $n=1$, and 
\cite{NToV} in general). 
We further remark that
there exist sets that are ADR (and that even form the boundary of a domain satisfying 
interior Corkscrew and Harnack Chain conditions),
but that are totally non-rectifiable (e.g., see the construction of Garnett's ``4-corners Cantor set"
in \cite[Chapter1]{DS2}).  Finally, we mention that there are numerous other characterizations of UR sets
(many of which remain valid in higher co-dimensions); cf. \cite{DS1,DS2}.

\begin{definition} \label{bpdef} ({\bf $BP(\mathcal{S})$ and $BP^2(LG)$}). 
Let $\mathcal{S}$ be a collection of subsets of $\ree$. We say an $n$-dimensional $ADR$ set $E \subset \ree$ has big pieces of $\mathcal{S}$ (``$E \in BP(\mathcal{S})$") if there exists a positive constant $\theta$ 
such that for each $x \in E$ and $r \in (0,\diam(E))$, there is a set $S \in \mathcal{S}$ with
$$H^n(B(x,r) \cap E \cap S) \ge \theta \  r^n.$$

A Lipschitz graph in $\ree$ is a set of the form 
$$\{ y + \rho(y) : y \in P\} $$
where $P$ is an $n$-plane and $\rho$ is a Lipschitz mapping onto a line perpendicular to P.
We say that $E$ has big pieces of Lipschitz graphs (``$BP(LG)$'') if there exists a positive constant $M_0$ such that $E \in BP(\mathcal{S})$, where $\mathcal{S}$ is the collection of all Lipschitz graphs with Lipschitz constant no greater than $M_0$.

Finally, if $E$ has $BP(\mathcal{S})$, where $\mathcal{S}$ is a collection of sets satisfying $BP(LG)$, with 
uniform bounds on the various constants, then we say that $E \in BP^2(LG)$.

\end{definition}

\begin{definition}\label{defurchar} ({\bf ``UR character"}).   Given a UR set $E\subset \ree$, its ``UR character"
is just the pair of constants $(\theta,M_0)$ involved in the definition of uniform rectifiability,
along with the ADR constant; or equivalently,
the quantitative bounds involved in any particular characterization of uniform rectifiability.
\end{definition}

\subsection{Further Notation and Definitions}

\begin{list}{$\bullet$}{\leftmargin=0.4cm  \itemsep=0.2cm}

\item We use the letters $c,C$ to denote harmless positive constants, not necessarily
the same at each occurrence, which depend only on dimension and the
constants appearing in the hypotheses of the theorems (which we refer to as the
``allowable parameters'').  We shall also
sometimes write $a\lesssim b$ and $a \approx b$ to mean, respectively,
that $a \leq C b$ and $0< c \leq a/b\leq C$, where the constants $c$ and $C$ are as above, unless
explicitly noted to the contrary.  At times, we shall designate by $M$ a particular constant whose value will remain unchanged throughout the proof of a given lemma or proposition, but
which may have a different value during the proof of a different lemma or proposition.

\item Given a closed set $E \subset \ree$, we shall
use lower case letters $x,y,z$, etc., to denote points on $E$, and capital letters
$X,Y,Z$, etc., to denote generic points in $\ree$ (especially those in $\ree\setminus E$).

\item The open $(n+1)$-dimensional Euclidean ball of radius $r$ will be denoted
$B(x,r)$ when the center $x$ lies on $E$, or $B(X,r)$ when the center
$X \in \ree\setminus E$.  A ``surface ball'' is denoted
$\Delta(x,r):= B(x,r) \cap E.$

\item Given a Euclidean ball $B$ or surface ball $\Delta$, its radius will be denoted
$r_B$ or $r_\Delta$, respectively.

\item Given a Euclidean or surface ball $B= B(X,r)$ or $\Delta = \Delta(x,r)$, its concentric
dilate by a factor of $\kappa >0$ will be denoted
$\kappa B := B(X,\kappa r)$ or $\kappa \Delta := \Delta(x,\kappa r).$

\item Given a (fixed) closed set $E \subset \ree$, for $X \in \ree$, we set $\delta(X):= \dist(X,E)$.

\item We let $H^n$ denote $n$-dimensional Hausdorff measure, and let
$\sigma := H^n\big|_{E}$ denote the ``surface measure'' on a closed set $E$
of co-dimension 1.

\item For a Borel set $A\subset \ree$, we let $1_A$ denote the usual
indicator function of $A$, i.e. $1_A(x) = 1$ if $x\in A$, and $1_A(x)= 0$ if $x\notin A$.

\item For a Borel set $A\subset \ree$,  we let $\interior(A)$ denote the interior of $A$.


\item We shall use the letter $I$ (and sometimes $J$)
to denote a closed $(n+1)$-dimensional Euclidean dyadic cube with sides
parallel to the co-ordinate axes, and we let $\ell(I)$ denote the side length of $I$.
If $\ell(I) =2^{-k}$, then we set $k_I:= k$.
Given an ADR set $E\subset \ree$, we use $Q$ to denote a dyadic ``cube''
on $E$.  The
latter exist (cf. \cite{DS1}, \cite{Ch}), and enjoy certain properties
which we enumerate in Lemma \ref{lemmaCh} below.

\end{list}

\begin{definition} ({\bf Corkscrew point}).  \label{def1.cork}
Following
\cite{JK}, given an open set $\Omega\subset \ree$, and a ball $B=B(x,r)$, with
$x\in \partial\Omega$ and
$0<r<\diam(\partial\Omega)$, we say that a point $X = X_B \in \om$ is a
 {\it Corkscrew point relative to $B$ with constant} $c>0$, if
there is a ball
$B(X,cr)\subset B(x,r)\cap\Omega$.  
\end{definition}

\begin{definition}\label{def2.cork}({\bf 2-sided Corkscrew condition}).  We say that an 
open set $\om$ satisfies the 2-sided Corkscrew condition 
if for some uniform constant $c > 0$ (the ``Corkscrew constant"),
and for every $x \in \partial\om$ and $0 < r < \diam(\pom)$, there are 
two Corkscrew points $X_1$ and $X_2$ relative to $B(x,r)$, with constant $c$  (as in Definition \ref{def1.cork}),
and two distinct connected components of $\rn \setminus \partial\om$, 
$\mathcal{O}_1$ and $\mathcal{O}_2$, with $B_1 = B(X_1, cr)\subset\mathcal{O}_1$ and $
B_2 = B(X_2, cr)\subset\mathcal{O}_2$.
We recall that this property is called ``Condition B" in the work of
David and Semmes \cite{DS3} .  We refer to the balls $B_1$ and $B_2$ as Corkscrew balls.
\end{definition}

\begin{definition}\label{def3.cork}({\bf Strong 2-sided Corkscrew condition}).
We say that an open set $\om$ satisfies the strong 2-sided Corkscrew condition 
if $\om$ satisfies the 2-sided Corkscrew condition, 
and one of the balls
$B_1 \subset \om$ or $B_2 \subset \om$.
\end{definition}

\begin{definition}({\bf Harnack Chain condition}).  \label{def1.hc} Following \cite{JK}, we say
that $\Omega$ satisfies the Harnack Chain condition if there is a uniform constant $C$ such that
for every $\rho >0,\, \Lambda\geq 1$, and every pair of points
$X,X' \in \Omega$ with $\delta(X),\,\delta(X') \geq\rho$ and $|X-X'|<\Lambda\,\rho$, there is a chain of
open balls
$B_1,\dots,B_N \subset \Omega$, $N\leq C(\Lambda)$,
with $X\in B_1,\, X'\in B_N,$ $B_k\cap B_{k+1}\neq \emptyset$
and $C^{-1}\diam (B_k) \leq \dist (B_k,\partial\Omega)\leq C\diam (B_k).$  The chain of balls is called
a ``Harnack Chain''.
\end{definition}

\begin{definition}({\bf NTA}). \label{def1.nta} Again following \cite{JK}, we say that a
domain $\Omega\subset \ree$ is NTA (``Non-tangentially accessible") if it satisfies the
Harnack Chain condition, and  the strong 2-sided Corkscrew condition.
\end{definition}

\begin{definition}({\bf Chord-arc domain}). \label{def1.ca}  An NTA domain with an ADR boundary is said to be
a Chord-arc domain.
\end{definition}

\begin{definition}\label{defAinfty}
({\bf $A_\infty$ and weak-$A_\infty$}). \label{ss.ainfty}
Given an ADR set $E\subset\ree$, 
and a surface ball
$\Delta_0:= B_0 \cap E$,
we say that a Borel measure $\mu$ defined on $E$ belongs to
$A_\infty(\Delta_0)$ if there are positive constants $C$ and $\theta$
such that for each surface ball $\Delta = B\cap E$, with $B\subseteq B_0$,
we have
\begin{equation}\label{eq1.wainfty}
\mu (F) \leq C \left(\frac{\sigma(F)}{\sigma(\Delta)}\right)^\theta\,\mu (\Delta)\,,
\qquad \mbox{for every Borel set } F\subset \Delta\,.
\end{equation}
Similarly, $\mu \in$ weak-$A_\infty(\Delta_0)$, with $\Delta_0 = B_0 \cap \pom$, if for every $\Delta = B \cap \pom$ with $2B \subseteq B_0$ we have
\begin{equation}\label{eq1.weakainfty}
\mu (F) \leq C \left(\frac{\sigma(F)}{\sigma(\Delta)}\right)^\theta\,\mu (2\Delta)\,,
\qquad \mbox{for every Borel set } F\subset \Delta\,.
\end{equation}
In the case that $\mu =\hm$ is harmonic measure for an open set $\Omega$ satisfying an interior 
Corkscrew condition, setting $E=\pom$,
we shall say that $\hm$ belongs to $A_\infty$
(resp., weak-$A_\infty$), if for every surface ball $\Delta_0$, and for
any Corkscrew point $X_{\Delta_0}\in \Omega$ relative to $\Delta_0$,
harmonic measure $\hm^{X_{\Delta_0}}$,
with pole at $X_{\Delta_0}$, belongs to $A_\infty(\Delta_0)$ (resp., weak-$A_\infty(\Delta_0)$), 
in the sense above.
\end{definition}

\begin{definition} \label{ibpdef} ($IBP(GHME)$).
When the collection $\mathcal{S}$ in Definition \ref{bpdef}
consists of boundaries of domains $\oT\subset \ree\setminus E$,
for which the associated harmonic measures belong to weak-$A_\infty$, and if the various boundaries
$\{\partial\oT\}$ are ADR, with uniform control of the ADR and weak-$A_\infty$ constants, then we say that
$E$ has ``interior big pieces of good harmonic measure estimates", and we write $E\in IBP(GHME)$.
\end{definition}

\begin{lemma}\label{lemmaCh}\textup{({\bf Existence and properties of the ``dyadic grid''})
\cite{DS1,DS2}, \cite{Ch}.}
Suppose that $E\subset \ree$ is closed $n$-dimensional ADR set.  Then there exist
constants $ a_0>0,\, \gamma>0$ and $C_1<\infty$, depending only on dimension and the
ADR constant, such that for each $k \in \mathbb{Z},$
there is a collection of Borel sets (``cubes'')
$$
\mathbb{D}_k:=\{Q_{j}^k\subset E: j\in \mathfrak{I}_k\},$$ where
$\mathfrak{I}_k$ denotes some (possibly finite) index set depending on $k$, satisfying

\begin{list}{$(\theenumi)$}{\usecounter{enumi}\leftmargin=.8cm
\labelwidth=.8cm\itemsep=0.2cm\topsep=.1cm
\renewcommand{\theenumi}{\roman{enumi}}}

\item $E=\cup_{j}Q_{j}^k\,\,$ for each
$k\in{\mathbb Z}$.

\item If $m\geq k$ then either $Q_{i}^{m}\subset Q_{j}^{k}$ or
$Q_{i}^{m}\cap Q_{j}^{k}=\emptyset$.

\item For each $(j,k)$ and each $m<k$, there is a unique
$i$ such that $Q_{j}^k\subset Q_{i}^m$.

\item $\diam\big(Q_{j}^k\big)\leq C_1 2^{-k}$.

\item Each $Q_{j}^k$ contains some ``surface ball'' $\Delta \big(x^k_{j},a_02^{-k}\big):=
B\big(x^k_{j},a_02^{-k}\big)\cap E$.

\item $H^n\big(\big\{x\in Q^k_j:{\rm dist}(x,E\setminus Q^k_j)\leq \varrho \,2^{-k}\big\}\big)\leq
C_1\,\varrho^\gamma\,H^n\big(Q^k_j\big),$ for all $k,j$ and for all $\varrho\in (0,a_0)$.
\end{list}
\end{lemma}

A few remarks are in order concerning this lemma.

\begin{list}{$\bullet$}{\leftmargin=0.4cm  \itemsep=0.2cm}

\item In the setting of a general space of homogeneous type, this lemma has been proved by Christ
\cite{Ch}, with the
dyadic parameter $1/2$ replaced by some constant $\delta \in (0,1)$.
In fact, one may always take $\delta = 1/2$ (see  \cite[Proof of Proposition 2.12]{HMMM}).
In the presence of the Ahlfors-David
property (\ref{eq1.ADR}), the result already appears in \cite{DS1,DS2}.

\item  For our purposes, we may ignore those
$k\in \mathbb{Z}$ such that $2^{-k} \gtrsim {\rm diam}(E)$, in the case that the latter is finite.

\item  We shall denote by  $\mathbb{D}=\mathbb{D}(E)$ the collection of all relevant
$Q^k_j$, i.e., $$\mathbb{D} := \cup_{k} \mathbb{D}_k,$$
where, if $\diam (E)$ is finite, the union runs
over those $k$ such that $2^{-k} \lesssim  {\rm diam}(E)$.

\item Properties $(iv)$ and $(v)$ imply that for each cube $Q\in\mathbb{D}_k$,
there is a point $x_Q\in E$, a Euclidean ball $B(x_Q,r)$ and a surface ball
$\Delta(x_Q,r):= B(x_Q,r)\cap E$ such that
$r\approx 2^{-k} \approx {\rm diam}(Q)$
and \begin{equation}\label{cube-ball}
\Delta(x_Q,r)\subset Q \subset \Delta(x_Q,Cr),\end{equation}
for some uniform constant $C$.
We shall denote this ball and surface ball by
\begin{equation}\label{cube-ball2}
B_Q:= B(x_Q,r) \,,\qquad\Delta_Q:= \Delta(x_Q,r),\end{equation}
and we shall refer to the point $x_Q$ as the ``center'' of $Q$.

\item For each cube $Q \in \mathbb{D}$, we let $X_Q$ be a corkscrew point relative to $B_Q$, 
and refer to this as a corkscrew point relative to $Q$. Such a corkscrew point exists, since $E$ is  $n$-dimensional ADR
(with the constant $c$ in Definition \ref{def1.cork} depending only on dimension and the ADR constants).

\item For a dyadic cube $Q\in \mathbb{D}_k$, we shall
set $\ell(Q) = 2^{-k}$, and we shall refer to this quantity as the ``length''
of $Q$.  Evidently, $\ell(Q)\approx \diam(Q).$

\item For a dyadic cube $Q \in \mathbb{D}$, we let $k(Q)$ denote the ``dyadic generation''
to which $Q$ belongs, i.e., we set  $k = k(Q)$ if
$Q\in \mathbb{D}_k$; thus, $\ell(Q) =2^{-k(Q)}$.

\end{list}

\section{A bilateral corona decomposition and corona type approximation by chord arc domains}\label{s2}\label{s3}


In this section, we state a bilateral variant of the ``corona decomposition" of David and Semmes
\cite{DS1,DS2}. The bilateral version was proved in \cite{HMM2}, Lemma 2.2. 
We first recall the notions of ``coherency'' and ``semi-coherency'':

\begin{definition}\cite{DS2}.\label{d3.11}   
Let $\sbf\subset \dd(E)$. We say that $\sbf$ is
``coherent" if the following conditions hold:
\begin{itemize}\itemsep=0.1cm

\item[$(a)$] $\sbf$ contains a unique maximal element $Q(\sbf)$ which contains all other elements of $\sbf$ as subsets.

\item[$(b)$] If $Q$  belongs to $\sbf$, and if $Q\subset \widetilde{Q}\subset Q(\sbf)$, then $\widetilde{Q}\in {\bf S}$.

\item[$(c)$] Given a cube $Q\in \sbf$, either all of its children belong to $\sbf$, or none of them do.

\end{itemize}
We say that $\sbf$ is ``semi-coherent'' if only conditions $(a)$ and $(b)$ hold. 
\end{definition}

\smallskip

The bilateral ``corona decomposition'' is as follows.
\begin{lemma}\label{lemma2.1} \cite[Lemma 2.2]{HMM2}. Suppose that $E\subset \ree$ is $n$-dimensional UR.  Then given any positive constants
$\eta\ll 1$
and $K\gg 1$, there is a disjoint decomposition
$\dd(E) = \G\cup\B$, satisfying the following properties.
\begin{enumerate}
\item  The ``Good"collection $\G$ is further subdivided into
disjoint stopping time regimes, such that each such regime {\bf S} is coherent (cf. Definition \ref{d3.11}).

\item The ``Bad" cubes, as well as the maximal cubes $Q(\sbf)$ satisfy a Carleson
packing condition:
$$\sum_{Q'\subset Q, \,Q'\in\B} \sigma(Q')
\,\,+\,\sum_{\sbf: Q(\sbf)\subset Q}\sigma\big(Q(\sbf)\big)\,\leq\, C_{\eta,K}\, \sigma(Q)\,,
\quad \forall Q\in \dd(E)\,.$$
\item For each $\sbf$, there is a Lipschitz graph $\Gamma_{\sbf}$, with Lipschitz constant
at most $\eta$, such that, for every $Q\in \sbf$,
\begin{equation}\label{eq2.2a}
\sup_{x\in \Delta_Q^*} \dist(x,\Gamma_{\sbf} )\,
+\,\sup_{y\in B_Q^*\cap\Gamma_{\sbf}}\dist(y,E) < \eta\,\ell(Q)\,,
\end{equation}
where $B_Q^*:= B(x_Q,K\ell(Q))$ and $\Delta_Q^*:= B_Q^*\cap E$.
\end{enumerate}
\end{lemma}

In this section, we construct the same domains as in \cite{HMM2}, for each stopping time regime $\sbf$
in Lemma \ref{lemma2.1}, a pair of NTA domains $\Omega_{\sbf}^\pm$, with ADR boundaries, which
provide a good approximation to $E$, at the scales within $\sbf$, in some appropriate sense.  To be a bit more precise,
$\Omega_{\sbf}:= \Omega_{\sbf}^+\cup \Omega_{\sbf}^-$ will be constructed as a sawtooth region
relative to some family of dyadic cubes, and the nature of this construction will be essential to the dyadic analysis that we will use below.
In this section, we follow essentially verbatim the construction in \cite{HMM2}, which 
we reproduce here for the reader's convenience.

We first discuss some preliminary matters.  We shall utilize the notation and constructions of
\cite{HMM2} (and essentially that of \cite{HM-I} and \cite{HMU}).

Let $\mathcal{W}=\W(\ree\setminus E)$ denote a collection
of (closed) dyadic Whitney cubes of $\ree\setminus E$, so that the cubes in $\mathcal{W}$
form a pairwise non-overlapping covering of $\ree\setminus E$, which satisfy
\begin{equation}\label{Whintey-4I}
4 \diam(I)\leq
\dist(4I,E)\leq \dist(I,E) \leq 40\diam(I)\,,\qquad \forall\, I\in \mathcal{W}\,\end{equation}
(just dyadically divide the standard Whitney cubes, as constructed in  \cite[Chapter VI]{St},
into cubes with side length 1/8 as large)
and also
$$(1/4)\diam(I_1)\leq\diam(I_2)\leq 4\diam(I_1)\,,$$
whenever $I_1$ and $I_2$ touch.


Let $E$ be an $n$-dimensional ADR set and pick two parameters $\eta\ll 1$ and $K\gg 1$. Define
\begin{equation}\label{eq3.1}
\W^0_Q:= \left\{I\in \W:\,\eta^{1/4} \ell(Q)\leq \ell(I)
 \leq K^{1/2}\ell(Q),\ \dist(I,Q)\leq K^{1/2} \ell(Q)\right\}.
 \end{equation}
 
\begin{remark}\label{remark:E-cks} 
We note that $\W^0_Q$ is non-empty,
 provided that we choose $\eta$ small enough, and $K$ large enough, depending only on dimension 
 and the ADR constant of $E$.
  \end{remark}

Assume now that $E$ is UR  and make the corresponding bilateral corona decomposition of Lemma \ref{lemma2.1} with $\eta\ll 1$ and $K\gg 1$.
Given $Q\in \dd(E)$, for this choice of $\eta$ and $K$, we set (as above) $B_Q^*:= B(x_Q, K\ell(Q))$, where we recall that $x_Q$ is the ``center" of $Q$ (see \eqref{cube-ball}-\eqref{cube-ball2}). For a fixed stopping time regime $\sbf$, we choose a co-ordinate system
 so that $\Gamma_{\sbf} =\{(z,\vp_{\sbf}(z)):\, z\in \rn\}$, where $\vp_{\sbf}:\rn\mapsto
 \re$ is a Lipschitz function with
 $\|\vp\|_{\rm Lip} \leq\eta$.

\begin{claim}\label{c3.1} If $Q\in \sbf$, and $I\in \W^0_Q$, then $I$ lies either above
or below $\Gamma_{\sbf}$.  Moreover,
$\dist(I,\Gamma_{\sbf})\geq \eta^{1/2} \ell(Q)$ (and therefore, by \eqref{eq2.2a},  
$\dist(I,\Gamma_{\sbf}) \approx \dist(I,E)$, with  implicit constants that may depend on $\eta$
and $K$).
\end{claim}

\begin{proof}[Proof of Claim \ref{c3.1}]
Suppose by way of contradiction that $\dist(I,\Gamma_{\sbf}) \leq \eta^{1/2}\ell(Q)$.
Then we may choose $y\in \GS$ such that
$$\dist(I,y)\leq\eta^{1/2}\ell(Q)\,.$$
By construction of $\W^0_Q$, it follows that for all $Z\in I$,
$|Z-y|\lesssim K^{1/2}\ell(Q)$.
Moreover, $|Z-x_Q|\lesssim K^{1/2}\ell(Q)$, and therefore
$|y-x_Q|\lesssim K^{1/2}\ell(Q)$.  In particular, $y\in B_Q^*\cap\GS$,
so by \eqref{eq2.2a}, $\dist(y,E)\leq\eta\,\ell(Q)$.  On the other hand,
choosing $Z_0\in I$  such that $|Z_0-y|=\dist(I,y) \leq\eta^{1/2}\ell(Q)$, we obtain
$\dist(I,E)\leq 2\eta^{1/2}\ell(Q)$.  For $\eta$ small, this contradicts the
Whitney construction, since $\dist(I,E) \approx \ell(I) \geq \eta^{1/4}\ell(Q)$.
\end{proof}

Next, given $Q\in\sbf$,  we augment $\W^0_Q$.  We split $\W^0_Q=\W_Q^{0,+} \cup \W_Q^{0,-}$,
where $I\in\W_Q^{0,+}$ if $I$ lies above
$\GS$, and  $I\in\W_Q^{0,-}$ if $I$ lies below
$\GS$.   Choosing $K$ large and $\eta$ small enough, by \eqref{eq2.2a}, 
we may assume that both $\W_Q^{0,\pm}$ are non-empty.
We focus on $\W_Q^{0,+}$, as the construction for $\W_Q^{0,-}$ is the same.
For each $I\in \W_Q^{0,+}$, let $X_I$ denote the center of $I$.
Fix one particular $I_0\in \W_Q^{0,+}$, with center $X^+_Q:= X_{I_0}$.
Let $\widetilde{Q}$ denote the dyadic parent of $Q$, unless $Q=Q(\sbf)$; in the latter case
we simply set $\Qt=Q$.  Note that
$\widetilde{Q}\in\sbf$,
by the coherency of $\sbf$.
By Claim \ref{c3.1}, for each $I$ in $\W_Q^{0,+}$, or in   $\W_{\widetilde{Q}}^{0,+}$,
we have
$$\dist(I,E)\approx\dist(I,Q)\approx\dist(I,\GS)\,,$$
where the implicit constants may depend on $\eta$ and $K$.  Thus, for each
such $I$, we may fix a Harnack chain, call it
$\mathcal{H}_I$, relative to the Lipschitz domain
$$\Omega_{\GS}^+:=\left\{(x,t) \in \ree: t>\vp_{\sbf}(x)\right\}\,,$$
connecting $X_I$ to $X_Q^+$.  By the bilateral approximation condition
\eqref{eq2.2a}, the definition of $\W^0_Q$, and the fact that $K^{1/2}\ll K$,
we may construct this Harnack Chain so that it consists of a bounded
number of balls (depending on $\eta$ and $K$), and stays a distance at least
$c\eta^{1/2}\ell(Q)$ away from $\GS$ and from $E$.
We let $\W^{*,+}_Q$ denote the set of all $J\in\W$ which meet at least one of the Harnack
chains $\mathcal{H}_I$, with $I\in \W_Q^{0,+}\cup\W_{\widetilde{Q}}^{0,+}$
(or simply $I\in \W_Q^{0,+}$, if $Q=Q(\sbf)$), i.e.,
$$\W^{*,+}_Q:=\left\{J\in\W: \,\exists \, I\in \W_Q^{0,+}\cup\W_{\Qt}^{0,+}
\,\,{\rm for\, which}\,\, \mathcal{H}_I\cap J\neq \emptyset \right\}\,,$$
where as above, $\widetilde{Q}$ is the dyadic parent of $Q$, unless
$Q= Q(\sbf)$, in which case we simply set $\Qt=Q$ (so the union is redundant).
We observe that, in particular, each $I\in \W^{0,+}_Q\cup \W^{0,+}_{\Qt}$
meets $\mathcal{H}_I$, by definition, and therefore
\begin{equation}\label{eqW}
\W_Q^{0,+}\cup \W^{0,+}_{\Qt} \subset \W_Q^{*,+}\,.
\end{equation}
Of course, we may construct $\W^{*,-}_Q$ analogously.
We then set
$$\W^*_Q:=\W^{*,+}_Q\cup \W^{*,-}_Q\,.$$
It follows from the construction of the augmented collections $\W_Q^{*,\pm}$ 
that there are uniform constants $c$ and $C$ such that
\begin{eqnarray}\label{eq2.whitney2}
& c\eta^{1/2} \ell(Q)\leq \ell(I) \leq CK^{1/2}\ell(Q)\,, \quad \forall I\in \mathcal{W}^*_Q,\\\nonumber
&\dist(I,Q)\leq CK^{1/2} \ell(Q)\,,\quad\forall I\in \mathcal{W}^*_Q.
\end{eqnarray}

Observe that $\W_Q^{*,\pm}$ and hence also $\W^*_Q$
have been defined for any $Q$ that belongs to some stopping time regime $\sbf$,
that is, for any $Q$ belonging to the ``good" collection $\G$ of Lemma \ref{lemma2.1}.  On the other hand, we
have defined $\W_Q^0$ for {\it arbitrary} $Q\in \dd(E)$.

We now set
\begin{equation}\label{Wdef}
\W_Q:=\left\{
\begin{array}{l}
\W_Q^*\,,
\,\,Q\in\G,
\\[6pt]
\W_Q^0\,,
\,\,Q\in\B
\end{array}
\right.\,,
\end{equation}
and for $Q \in\G$ we shall henceforth simply write $\W_Q^\pm$ in place of $\W_Q^{*,\pm}$.

Next, we choose a small parameter $\tau_0>0$, so that
for any $I\in \W$, and any $\tau \in (0,\tau_0]$,
the concentric dilate
$I^*(\tau):= (1+\tau) I$ still satisfies the Whitney property
\begin{equation}\label{whitney}
\diam I\approx \diam I^*(\tau) \approx \dist\left(I^*(\tau), E\right) \approx \dist(I,E)\,, \quad 0<\tau\leq \tau_0\,.
\end{equation}
Moreover,
for $\tau\leq\tau_0$ small enough, and for any $I,J\in \W$,
we have that $I^*(\tau)$ meets $J^*(\tau)$ if and only if
$I$ and $J$ have a boundary point in common, and that, if $I\neq J$,
then $I^*(\tau)$ misses $(3/4)J$.
Given an  arbitrary $Q\in\dd(E)$, we may define an associated
Whitney region $U_Q$ (not necessarily connected), as follows:
\begin{equation}\label{eq3.3aa}
U_Q=U_{Q,\tau}:= \bigcup_{I\in \W_Q} I^*(\tau)
\end{equation}
For later use, it is also convenient to introduce some fattened version of $U_Q$:  if $0<\tau\le \tau_0/2$,
\begin{equation}\label{eq3.3aa-fat}
\widehat{U}_Q=U_{Q,2\,\tau}:= \bigcup_{I\in \W_Q} I^*(2\,\tau). 
\end{equation}
If $Q\in\G$, then $U_Q$ splits into exactly two connected components
\begin{equation}\label{eq3.3b}
U_Q^\pm= U_{Q,\tau}^{\pm}:= \bigcup_{I\in \W^{\pm}_Q} I^*(\tau)\,.
\end{equation}
When the particular choice of $\tau\in (0,\tau_0]$ is not important,
for the sake of notational convenience, we may
simply write $I^*$, $U_Q$,  and $U_Q^{\pm}$ in place of
 $I^*(\tau)$, $U_{Q,\tau}$, and $U_{Q,\tau}^{\pm}$.
We note that for $Q\in\G$, each $U_Q^{\pm}$ is Harnack chain connected, by construction
(with constants depending on the implicit parameters $\tau, \eta$ and $K$);
moreover, for a fixed stopping time regime $\sbf$,
if $Q'$ is a child of $Q$, with both $Q',\,Q\in \sbf$, then
$U_{Q'}^{+}\cup U_Q^{+}$
is Harnack Chain connected, and similarly for
$U_{Q'}^{-}\cup U_Q^{-}$.  

We may also define ``Carleson Boxes" relative to any $Q\in\dd(E)$, by
\begin{equation}\label{eq3.3a}
T_Q=T_{Q,\tau}:=\interior\left(\bigcup_{Q'\in\dd_Q} U_{Q,\tau}\right)\,,
\end{equation}
where
\begin{equation}\label{eq3.4a}
\dd_Q:=\left\{Q'\in\dd(E):Q'\subset Q\right\}\,.
\end{equation}
Let us note that we may choose $K$ large enough so that, for every $Q$,
\begin{equation}\label{eq3.3aab}
T_Q \subset B_Q^*:= B\left(x_Q,K\ell(Q)\right)\,. 
\end{equation}
For future reference, we also introduce dyadic sawtooth regions as follows.
Given a family $\mathcal{F}$ of disjoint cubes $\{Q_j\}\subset \mathbb{D}$, we define
the {\bf global discretized sawtooth} relative to $\F$ by
\begin{equation}\label{eq2.discretesawtooth1}
\dd_{\F}:=\dd\setminus \bigcup_{\F} \dd_{Q_j}\,,
\end{equation}
i.e., $\dd_{\F}$ is the collection of all $Q\in\dd$ that are not contained in any $Q_j\in\F$.
Given some fixed cube $Q$,
the {\bf local discretized sawtooth} relative to $\F$ by
\begin{equation}\label{eq2.discretesawtooth2}
\dd_{\F,Q}:=\dd_Q\setminus \bigcup_{\F} \dd_{Q_j}=\dd_\F\cap\dd_Q.
\end{equation}
Note that in this way $\dd_Q=\dd_{\textup{\O},Q}$.

Similarly, we may define geometric sawtooth regions as follows.
Given a family $\mathcal{F}$ of disjoint cubes $\{Q_j\}\subset \mathbb{D}$,
we define the {\bf global sawtooth} and the {\bf local sawtooth} relative to $\mathcal{F}$ by respectively
\begin{equation}\label{eq2.sawtooth1}
\Omega_{\mathcal{F}}:= {\rm int } \bigg( \bigcup_{Q'\in\dd_\F} U_{Q'}\bigg)\,,
\qquad
\Omega_{\mathcal{F},Q}:=  {\rm int } \bigg( \bigcup_{Q'\in\dd_{\F,Q}} U_{Q'}\bigg)\,.
\end{equation}
Notice that $\Omega_{\textup{\O},Q}=T_Q $.
For the sake of notational convenience, given a pairwise disjoint family $\F\in \dd$, and a cube $Q\in \dd_{\F}$, we set
\begin{equation}\label{Def-WF}
\W_{\F}:=\bigcup_{Q'\in\dd_{\F}}\W_{Q'}\,,\qquad
\W_{\F,Q}:=\bigcup_{Q'\in\dd_{\F,Q}}\W_{Q'}\,,\end{equation}
so that in particular, we may write
\begin{equation}\label{eq3.saw}
\Omega_{\mathcal{F},Q}={\rm int }\,\bigg(\bigcup_{I\in\,\W_{\F,Q}} I^*\bigg)\,.
\end{equation}


\smallskip

\begin{remark}\label{r3.11a}
We recall that, by construction
(cf. \eqref{eqW}, \eqref{Wdef}), $\W_{\Qt}^{0,\pm}\subset \W_Q$, and therefore
$Y_Q^\pm \in U_Q^\pm\cap U^\pm_{\Qt}$.  Moreover, since $Y_Q^\pm$ is the center of some $I\in \W_{\Qt}^{0,\pm}$,
we have that $\dist(Y_Q^\pm, \partial U_Q^\pm)\approx \dist(Y_Q^\pm, \partial U_{\Qt}^\pm) \approx \ell(Q)$
(with implicit constants possibly depending on
$\eta$ and/or $K$)
\end{remark}

\smallskip

\begin{remark}\label{r3.11}
Given a stopping time regime $\sbf$ as in Lemma \ref{lemma2.1}, for any semi-coherent
subregime (cf. Definition \ref{d3.11}) $\sbf'\subset \sbf$ (including, of course, $\sbf$ itself), we now set
\begin{equation}\label{eq3.2}
\Omega_{\sbf'}^\pm = {\rm int}\left(\bigcup_{Q\in\sbf'} U_Q^{\pm}\right)\,,
\end{equation}
and let $\Omega_{\sbf'}:= \Omega_{\sbf'}^+\cup\Omega^-_{\sbf'}$.
Note that implicitly, $\Omega_{\sbf'}$ depends upon $\tau$ (since $U_Q^\pm$ has such dependence).
When it is necessary to consider the value of $\tau$
explicitly, we shall write $\Omega_{\sbf'}(\tau)$.
\end{remark}

\smallskip
It is helpful to introduce some terminology now whose utility will become clear later. Let $Q \in \dd$ define the following

\begin{equation}\label{boundary whitney cubes} 
\mathcal{I}(Q) := \{ I \in \W: I \cap T_Q \neq \emptyset \} 
\end{equation}
and  also
\begin{equation}\label{buffer zone}
V(Q) = \interior\left( \bigcup_{I\in\mathcal{I}(Q)}I^*\right).
\end{equation}
We note that, trivially, $T_Q\subset V(Q)$. 
Notice also that if $\interior(I^\ast) \subset V(Q)$ then 
\begin{equation}
\label{property 1 of BQ} \dist(I^\ast, Q) \lesssim \ell(Q)
\end{equation}
and 
\begin{equation}
\label{property 2 of BQ} \ell(I) \lesssim \ell(Q).
\end{equation}

\medskip 
\begin{lemma} \label{separation of buffer zones}  Let $Q_1, Q_2 \in \dd(E)$ if $U_{Q_1}$ meets $U_{Q_2}$ then
\begin{equation} \label{Lemma 1 result 1} \dist(Q_1,Q_2) \lesssim \min\{\ell(Q_1),\ell(Q_2)\} \end{equation}
with implicit constant depending only on $K, \eta, $ and dimension.
Moreover there exists a constant $\Upsilon$ depending only on $K, \eta,$ and dimension such that if 
$$ \dist(Q_1,Q_2) > \Upsilon \max\{\ell(Q_1), \ell(Q_2)\}$$
then 
\begin{equation} \label{Lemma 1 result 2} \overline{V(Q_1)} \cap \overline{V(Q_2)} = \emptyset .\end{equation}
\end{lemma}
\begin{proof} Suppose that $U_{Q_1}$ meets $U_{Q_2}$ then we have that there exists a cube $I_1^\ast \in U_{Q_1}$ and $I_2^\ast \in U_{Q_2}$ such that $I_1^\ast \cap I_2^\ast \neq \emptyset$. Since $I_1^\ast$ and $I_2^\ast$ are Whitney cubes that meet we have that
\begin{equation} 
\label{Lemma 1 equation 1} \ell(I_1) \approx \ell(I_2).
\end{equation}
Then by construction of $U_{Q_1}$ and $U_{Q_2^\prime}$ we have
\begin{equation} 
\label{Lemma 1 equation 2} \ell(Q_1) \approx \ell(Q_2),
\end{equation}
\begin{equation} 
\label{Lemma 1 equation 3} \dist(Q_1, I_1^\ast) \lesssim \ell(Q_1), 
\end{equation}
\begin{equation} 
\label{Lemma 1 equation 4} \dist(Q_2, I_2^\ast) \lesssim \ell(Q_2^\ast). 
\end{equation}
So that \ref{Lemma 1 equation 1}, \ref{Lemma 1 equation 2}, \ref{Lemma 1 equation 3}, and \ref{Lemma 1 equation 4} yield 
$$\dist(Q_1,Q_2) \lesssim \ell(Q_1).$$
To prove \ref{Lemma 1 result 2} we need only see that by \ref{property 1 of BQ} and 
\ref{property 2 of BQ} if $\overline{V(Q_1)}$ meets $\overline{V(Q_2)}$
that 
\begin{equation}\label{Lemma 1 equation 5}
\dist(Q_1,Q_2) \lesssim \max\{\ell(Q_1), \ell(Q_2)\},
\end{equation}
in addition we can even put a distance between $V(Q_1)$ and $V(Q_2)$ on the order of $ \max\{\ell(Q_1), \ell(Q_2)\}$ by making $\Upsilon$ larger.
\end{proof}

\section{Carleson measures: proof of the Theorem \ref{t1}}
The proof will utilize the method of  ``extrapolation of Carleson measures". 
This method was first used by J. L. Lewis \cite{LM}, whose work was influenced by the 
Corona construction of Carleson \cite{Car} and the 
work of Carleson and Garnett \cite{CG} 
(see also \cite{HL}, \cite{AHLT}, \cite{AHMTT}, \cite{HM-TAMS}, \cite{HM-I}.) 
We will apply this method to the (discrete) packing measure from the bilateral Corona decomposition.
Let $E\subset \ree$ be a UR set of co-dimension 1.    We fix  positive numbers $\eta\ll 1$,
and $K\gg 1$, and for these values of $\eta$ and $K$,
we perform the bilateral Corona decomposition of $\dd(E)$ guaranteed by Lemma \ref{lemma2.1}.
Let $\M:=\{Q(\sbf)\}_{\sbf}$ denotes the collection of cubes which are the maximal elements of the stopping time regimes in $\G$.
Given a cube $Q\in \dd(E)$, we set
\begin{equation}\label{eq4.0}
\alpha_Q:= 
\begin{cases} \sigma(Q)\,,&{\rm if}\,\, Q\in \M\cup\B, \\
0\,,& {\rm otherwise}.\end{cases}
\end{equation}
Given  any collection $\dd'\subset\dd(E)$, we define
\begin{equation}\label{eq4.1}
\mut(\dd'):= \sum_{Q\in\dd'}\alpha_{Q}.
\end{equation}
We recall that $\dd_Q$ is the ``discrete Carleson region
relative to $Q$",
defined in \eqref{eq3.4a}. 
Then by Lemma \ref{lemma2.1} (2), we have the discrete Carleson measure estimate
\begin{multline}\label{eq4.3x}
\mut(\dd_Q):=\sum_{Q'\subset Q, \,Q'\in\B} \sigma(Q')
\,\,+\,\sum_{\sbf: Q(\sbf)\subset Q}\sigma\big(Q(\sbf)\big)\,\leq\, C_{\eta,K}\, \sigma(Q)\,,
\\[4pt] \forall Q\in \dd(E)\,.
\end{multline}
Given a family $\F:=\{Q_j\}\subset \dd(E)$ of
pairwise disjoint cubes, we recall that the ``discrete  sawtooth" $\dd_\F$ is 
the collection of all cubes in $\dd(E)$ that are not contained in any $Q_j\in\F$ (cf. \eqref{eq2.discretesawtooth1}),
and we define the ``restriction of $\mut$ to the sawtooth $\dd_\F$'' by
\begin{equation}\label{eq4.4x}
\mut_\F(\dd'):=\mut(\dd'\cap\dd_\F)= \sum_{Q\in\dd'\setminus (\cup_{\F} \,\dd_{Q_j})}\alpha_{Q}.
\end{equation}
We take the usual definition of Carleson norm
$$\|\mut\|_{\C}:= \sup_{Q\in\dd(E)}\frac{\mut(\dd_Q)}{\sigma(Q)}.$$
Notice that the way that we have defined $\mut$ we have 
\begin{equation} \|\mut\|_{\C} \le C_{\eta,K} \end{equation}
The following Lemma will be one of two crucial lemmas in proving Theorem \ref{t1}.
\begin{lemma}\label{lemma:Corona}{\cite[Lemma 7.2]{HM-I}}
Suppose that $E$ is ADR.  Fix $Q\in \dd(E)$
and $\mut$ as above.  Let $a\geq 0$ and $b>0$, and suppose that
$\mut(\dd_{Q})\leq (a+b)\,\sigma(Q).$
Then there is a family $\F=\{Q_j\}\subset\dd_{Q}$
of pairwise disjoint cubes, and a constant $C$ depending only on dimension
and the ADR constant such that
\begin{equation} \label{Corona-sawtooth}
\|\mut_\F\|_{\C(Q)}
\leq C b,
\end{equation}
\begin{equation}
\label{Corona-bad-cubes}
\sigma(B)
\leq \frac{a+b}{a+2b}\, \sigma(Q)\,,
\end{equation}
where $B$ is the union of those $Q_j\in\F$ such that
$\mut\big(\dd_{Q_j}\setminus \{Q_j\}\big)>a\,\sigma(Q_j)$.
\end{lemma}
The other crucial lemma is the following.
\begin{lemma}\label{subregime sawtooths are NTA}\cite[Lemma 3.24]{HMM2} : \label{lemma3.15} Let $\sbf$ be a given
stopping time regime as in Lemma \ref{lemma2.1},
and let $\sbf'$ be any nonempty, semi-coherent  subregime of $\sbf$.
Then for $0<\tau\leq\tau_0$, with $\tau_0$ small enough,
each of $\Omega^\pm_{\sbf'}$ is an NTA domain, with ADR boundary.
The constants in the NTA and ADR conditions depend only on $n,\tau,\eta, K$, and the
ADR/UR constants for $E$.
\end{lemma}
The following standard covering type lemma will be required.
\begin{lemma}\label{covering lemma} Fix $Q_0 \in \dd(E)$ and let $\F = \{Q_j\} \subset \dd_{Q_0}$ be any pairwise disjoint family of cubes.
Then for any positive constant $\kappa$ we may find a sub-collection $\G = \{\Qt_i\} \subset \F$ with the following properties:
\begin{equation}\label{covering property 1}
\sigma(\cup_{\mathcal{G}} \widetilde{Q}_i) \ge C \sigma(\cup_\mathcal{F} Q_j) 
\end{equation}
\begin{equation}
\label{covering property 2} \dist(\widetilde{Q}_i, \widetilde{Q}_k) \ge \kappa \max \{\ell(\widetilde{Q}_i), \ell(\widetilde{Q}_k))
\end{equation}
where $C$ depends on $\kappa$, dimension and ADR.
\end{lemma}

\begin{proof}[Proof of Theorem \ref{t1}]
First we fix $\eta$ and $K$ so that that Lemma \ref{subregime sawtooths are NTA} holds. The proof will follow by induction. For any $a \ge 0$ we have the induction hypothesis $H(a)$, defined in the following way. \\ \\
\noindent
\textbf{$H(a):$} There exists $\eta_a > 0$ such that for all $Q_0 \in \dd(E)$ satisfying $\mut(\dd_{Q_0}) \le a \sigma(Q)$, there is a collection $\dd' \subseteq \dd_{Q_0}$, 
and an open set $\oT$  of the form
\begin{equation}
\label{induction region} \oT := \interior\left(\cup_{Q_j \in \dd'} U_{Q_j}\right) \subset T_{Q_0}\,,
\end{equation}
which has an ADR boundary and satisfies the strong 2-sided Corkscrew condition for open sets with Corkscrew balls lying in $V(Q_0)$ (see \eqref{buffer zone}),  and in addition 
\begin{equation}
\label{big pieces of omega} \sigma(\partial\oT \cap Q_0) \ge \eta_a \sigma(Q_0).
\end{equation}
Moreover, each connected component of $\oT$ is an NTA domain with ADR boundary with uniform constants possibly depending on $a$.

To prove the Theorem it is enough to show that 
$H(M_0)$ holds where $M_0$ is the Carleson norm of $\mut$, that is, $M_ 0 = C_{\eta,K}$ in \ref{lemma2.1}. 
We do this by showing first that $H(0)$ holds and 
then that $H(a)$ implies $H(a+b)$ for some fixed constant $b$ depending only on dimension and the ADR
constants. This way we will only use finitely many steps to get to $H(M_0)$. 
We set $b = \frac{\gamma}{C}$ where $C$ is the constant in \ref{Corona-sawtooth} and 
$\gamma$ is a small positive number to be chosen.

The fact that $H(0)$ holds is somewhat trivial since this would imply that $Q_0$ and all of its descendants are in the same stopping time regime and therefore $Q_0$ coincides with a Lipschitz graph. Here we can also directly apply the results in \cite{HMM2}. We are then left with showing that $H(a)$ implies $H(a + b)$. 

\smallskip

\noindent{\bf Proof that $H(a) \implies H(a+b)$:} Suppose that $H(a)$ holds and that $Q_0 \in  \dd(E)$ be such that $\mut(\dd_{Q_0}) \le (a + b) \sigma(Q_0)$. First let $C_2$ be an integer so large that if $Q_1 \subseteq Q_2$ with $k(Q_2) + C_2 - 5 < k(Q_1) $ then we have that $U_{Q_2} \cap V(Q_1) = \emptyset$. We obtain via 
Lemma \ref{lemma:Corona} a collection $\F = \{Q_j \}_{j=1}^\infty$ such that 
\begin{equation}\label{eq4.15}
 \|\mut_\F \|_{C(Q_0)} \le Cb = \gamma 
 \end{equation}
and
\begin{equation}\label{eq4.16}
\sigma(B) \le \frac{a+b}{a+2b} \sigma(Q_0)\,,
\end{equation}
where $B$ is the union of those $Q_j$ in $\mathcal{F}$ such that 
$\mathfrak{m}(\mathbb{D}_{Q_j} \setminus \{Q_j\}) > a \sigma(Q_j)$, call this collection $\mathcal{F}_{bad}$.
Define $\mathcal{F}_{good} := \mathcal{F} \setminus \mathcal{F}_{bad}$.  
Then by pigeon-holing, 
for each $Q_j \in \mathcal{F}_{good}$ we may find a child of $Q_j$ to which we can 
apply the induction hypothesis $H(a)$. Iterating the pigeon-holing argument,
we may find a cube $Q_j^\prime$ that is $C_2$ generations down from $Q_j$
(i.e., so that $\ell(Q_j') = 2^{-C_2} \ell(Q_j)$), 
to which we may apply the induction hypothesis.

\begin{remark}\label
{remark4.16}
Choosing $\gamma$, small enough  in \eqref{eq4.15}
(in fact, $\gamma = 1/2$ will suffice), 
we obtain that
$\dd_{\F,Q_0}$ does not contain any $Q^\prime \in \M \cup \B$, where $\M := \{Q(\sbf)\}_\sbf$.
Fixing such a $\gamma$, we find therefore that
every $Q'\in \dd_{\mathcal{F},Q_0}$ is a good cube, and moreover all such 
$Q'$ belong to the same  ${\bf S}$, a
stopping time regime as in Lemma \ref{lemma2.1}. 
Thus, $ \dd_{\mathcal{F},Q_0}$ is a semi-coherent
 (see Definition \ref{d3.11})  subregime  of that {\bf S}, and therefore
$\om_{\mathcal{F},Q_0}$ splits into two disjoint NTA domains with ADR boundary,
 by Lemma \ref{subregime sawtooths are NTA}. 
\end{remark} 

 We may clearly assume that $a < M_0$. Set $\eta :=  1 - \frac{M_0 + b}{M_0 + 2b}$ and $A:= Q \setminus (\cup_\mathcal{F} Q_j)$ and $G : = (\cup_\mathcal{F} Q_j) \setminus B$. Then we see immediately that 
$$\sigma(A \cup G) \ge \eta \sigma(Q_0).$$
If $\sigma(A) > (\eta/2) \sigma(Q_0)$, we set $\oT = \om_{\mathcal{F},Q_0}$,
and note that  $A \subseteq \partial\oT$.   Then $H(a+b)$ holds in this case, by Remark \ref{remark4.16}.

Therefore it is enough to consider the case when $\sigma(G) \ge (\eta/2) \sigma(Q_0)$. 
Suppose first that  $\mathcal{F}= \{Q_0\}$.   Then necessarily, $Q_0\in \F_{good}$,
and in this case we may apply the induction hypothesis to a child of $Q_0$ to see that 
$H(a+ b)$ holds  
(possibly with smaller constant in the Corkscrew condition and larger ADR constant.)
Thus, we may assume that the collection $\mathcal{F} \neq \{Q_0\}$.
We now apply Lemma \ref{covering lemma} to $\F_{good}$ with $\kappa \ge \Upsilon$, where $\kappa$ is
 to be chosen momentarily, and where
$\Upsilon$ is the constant in Lemma \ref{separation of buffer zones}, 
to obtain a 
subcollection $\widetilde{\F} \subset \F_{good}$ with the following properties:
\begin{align*}
\sigma(\cup_{\widetilde{\mathcal{F}}} \,Q_i) &\gtrsim \sigma(\cup_{\mathcal{F}_{good}} Q_j) \ge  (\eta/2) \sigma(Q_0)  \\
\dist(Q_i, Q_k) &\ge \kappa \max \{\ell(Q_i), \ell(Q_k)\}\,, 
\quad \forall Q_i,Q_k \in \widetilde{\mathcal{F}},\quad i \neq k.
\end{align*}

Now  for each $Q_j \in \widetilde{\F}$, we define two families as follows:  
let $Q_j^\ast$ be the parent of $Q_j$, and  let $Q'_j$ be the cube $C_2$ generations down to which we can apply the induction hypothesis. Now set $\F ' := \{Q'_j\}_{Q_j\in \widetilde{\F} }$, and $\F^\ast 
:= \{Q^\ast_j \}_{Q_j\in\widetilde{\F}}$. Notice first that all of the cubes $Q_j^\ast$ are in 
$\dd_{\F,Q_0}$ and that $\F '$ has the same properties as $\widetilde{\F}$ namely
\begin{equation}\label{little domains are a big portion}
\sigma(\cup_{\F'} Q_i^\prime)\, \gtrsim\, \sigma(\cup_{\mathcal{F}_{good}} Q_j)\, 
\gtrsim \,(\eta/2) \sigma(Q_0)  
\end{equation}
and
\begin{equation} \label{separation Q primes}
\dist(Q'_i,Q'_k) \ge \kappa \max \{\ell(Q'_i), \ell(Q'_k)\}\,, 
\quad \forall Q_i^\prime,Q_k^\prime \in \mathcal{F}',\quad i \neq k.
\end{equation}
For each $Q'_j \in \F'$ we apply the induction hypothesis to obtain an open set as in $H(a)$ and 
call this set $\oT_j$.

Next, we construct two 	``large" NTA domains with ADR boundary with the help of Lemma 3.24 in [HMM]. For each $Q'_j \in \mathcal{F}'$, let $\I_j$ be the collection of fattened Whitney 
cubes $I^*$ such that $\interior(I^\ast) \subset \om_{\F, Q_0}$, and $I^*$ meets $V(Q'_j)$. Let $\mathcal B_j$ be the 
collection of $Q\in \dd_{Q_0}$ such that there 
exists an $I^\ast \subset U_Q$ with $I^\ast \in \I_j$. Now 
we define $\mathcal{F}_\infty$ as the cubes in $\mathcal{F} \cup (\cup_j \mathcal{B}_j)$ 
which are maximal with respect to containment.

By construction (see Remark \ref{remark4.16}), 
 $\sbf' := \dd_{\F_\infty, Q_0}$ is a semi-coherent subregime of some stopping time regime
 ${\bf S}$
 as in Lemma \ref{lemma2.1}.
 Thus, setting $\oT_0 = \om_{F_\infty, Q_0}$, by Lemma \ref{lemma3.15} we obtain that 
 $\om_0$ is the union of two disjoint NTA domains with ADR boundaries,  
 whose diameters are comparable to $\ell(Q_0)$. 
 
 We will need to know 
 that we did not remove too many cubes, namely, 
 we do not want to remove any $Q_j^\ast \in \F^\ast$.
\begin{claim}If $\kappa$ is chosen large enough then for
every $j \ge 1$ we have for each $Q_j^\ast \in \mathcal{F}^*$ that $Q_j^\ast \in \dd_{\F_\infty, Q_0}$.
\end{claim}

\begin{proof}[Proof of claim:] Note that by our choice of $C_2$, $U_{Q}$ does not meet $V(Q'_j)$, for 
any $Q\supseteq Q_j^*$.  Moreover, by construction, we have that $Q^*_j \notin \mathcal{F}$. Suppose now for the purposes of contradiction, that there is a $j \ge 1$ such that $Q^*_j \subseteq Q$, with $Q \in \F_\infty$. Then there exists a $k \neq j$ such that $Q \in \B_k$ and $U_{Q}$ meets $V(Q'_k)$ for some $Q'_k \in \mathcal{F}'$. Then we have immediately that $\ell(Q) \lesssim \ell(Q'_k) < \ell(Q_k)$ and $\dist(Q'_k,Q) \lesssim \ell(Q'_k) < \ell(Q_k)$ so that
$$\dist(Q'_k, Q'_j) \lesssim \dist(Q'_k, Q) + \ell(Q) \lesssim \ell(Q_k)$$
a contradiction for $\kappa$ large enough. We remind the reader that we are only considering the cubes extracted with a covering lemma so that they separated (see Lemma \ref{covering lemma} and (\ref{separation Q primes})).
\end{proof}

 \begin{remark}\label{remark4.18} 
We note for future reference that 
$V(Q_j')\cap  (\oT_0\cup\partial\oT_0) =\emptyset$, for every $j\geq1$.  Indeed, this follows by our observation, in
the preceding paragraph, that by choice of $C_2$ large enough,
$U_{Q}$ does not meet $V(Q'_j)$, for 
any $Q\supseteq Q_j^*$.   We further note that $\overline{V(Q_j')}\cap \overline{V(Q_k')}=\emptyset$, 
for all $j\neq k$, with $j,k\geq 1$,
by \eqref{separation Q primes}, Lemma \ref{separation of buffer zones}, and our choice of
$\kappa\geq \Upsilon$.
\end{remark}

Recall that for $j\geq 1$, $\oT_j$ is the open set associated to $Q'_j\in \F'$ via the induction hypothesis, 
and $\oT_0= \om_{F_\infty, Q_0}$. 
We now set $\oT := \cup_{j=0}^\infty \oT_j$. 
We shall show that $\oT$ has all the desired properties. First we show that $\oT$ satisfies a 2-sided Corkscrew condition. Note that by  Remark \ref{remark4.18},  the distinct components of $\oT_j$ remain distinct components in $\oT$. We let $x \in \partial\oT$ and $0< r \le \ell(Q_0)$,  let $M$ be a large number to be chosen, and set $\delta(x) = \dist(x,E)$. Since $\partial\oT \subseteq \cup_{j=0}^\infty \partial \oT_j$ we break into two cases:

\smallskip

\noindent\textbf{Case 1:} $x \in \partial\oT_0$.   
Recall that $\oT_0$ splits into two NTA domains  $\oT_0^\pm$.
Moreover, following the construction in  \cite{HMM2}, the
interior and exterior corkscrew 
points for the domain $\oT_0^+$ are found as follows.
Without loss of generality we may assume that $x \in \partial\oT_0^+$.
For M sufficiently large (depending 
only on allowable constants) the argument distinguishes between 
two cases, when $r < M \delta(x)$  and when $r \ge M\delta(x)$. In the case that $r \ge M\delta(x)$ we find one Corkscrew point in the domain $\oT_0^-$ and one Corkscrew point in $\oT_0^+$,
and these serve as Corkscrew points in separate components for $\oT$ as well. 
In the case that $r < M\delta(x)$, 
we have that $\delta(x) > 0$, and $x$ lies on the face of a fattened 
Whitney cube $I^*$ whose interior lies in $\oT_0^+$. 
Moreover, $x \in J$, for some 
Whitney cube $J \not\in (\cup_{Q \in \dd_{\mathcal{F}_\infty, Q_0}} \W_Q)$;  
we then have one 
Corkscrew point in $I^*$, 
and  a second in $J \setminus \oT_0^+$.   Clearly, the first of these is also a Corkscrew point
for $\oT$, in the component $\oT_0^+$.  To see that the second is a Corkscrew point relative
to $\partial \oT$, 
it remains to show that 
$J$ misses $\partial \oT\setminus \partial \oT_0^+$.  If not,
then $J$ must  intersect $\partial\oT_j$ 
for some $j\geq 1$, 
and therefore $T_{Q'_j}$ meets $J$, so that $J \subset V(Q'_j)$.   On the other hand, $J$
also meets $I^\ast$, so that $I^\ast \in \I_j$, hence there is a cube $Q\in \dd_{\mathcal{F}_\infty, Q_0}$
that belongs to $\mathcal{B}_j$, a contradiction.

\smallskip

\noindent
\textbf{Case 2:} $x \in \partial\oT_j$ for $j \ge 1$. Here $\oT_j$ is associated to $Q'_j$. 

\smallskip

\noindent \textbf{Case 2a:} $r \ge M \ell(Q_j^\prime)$. In this case,
since $Q^*_j \in \dd_{\F_\infty, Q_0}$, 
we have that $\text{int}(U^\pm_{Q^\ast_j}) \subset \oT_0^\pm$, 
so for $M$ large enough depending on $C_2$,
there exists  points $y^\pm \in \partial \oT^\pm_0$, with $|x-y^\pm|<r/2$.
 We may then apply Case 1 to each of the balls $B(y^\pm,r/2)\subset B(x,r)$.

\smallskip

\noindent \textbf{Case 2b:} $r < M\ell(Q'_j)$. 
From the induction hypothesis we have two Corkscrew 
balls $B_1$ and $B_2$ 
that satisfy the strong 2-sided 
Corkscrew condition, at scale $r/M$,  relative to the open set $\oT_j$
(see Definition \ref{def3.cork}).
Without loss of generality we may assume that $B_1 \subset \oT_j$, hence also
$B_1\subset \oT$.
We must show that these Corkscrew balls 
satisfy the strong 2-sided Corkscrew condition for 
open sets, with $\oT$ the open set in question, thus, it remains to show that
$B_1, B_2 \subset \ree\setminus \partial\oT$.   To this end, we simply observe that
$B_1$ and $B_2$ do not meet $\partial\oT_j$ by hypothesis, nor 
they do not meet $\partial\oT_0$ or $\partial\oT_i$ for some $i \neq j$, by Remark
\ref{remark4.18}, since
by the induction hypothesis $B_1, B_2$ are in $V(Q'_j)$, and since $\partial\oT_i\subset \overline{V(Q'_i)}$,
by construction.

\smallskip  

Next, we observe that $\partial \oT_0$ is ADR, by \cite[Appendix A]{HMM2}, and that
each $\partial\oT_j$ is ADR, with uniform control of the
ADR constants,  by the induction hypothesis.  Thus,
we are left with showing that $\oT$ has ADR boundary and that condition \ref{big pieces of omega} holds. 
We begin by verifying the upper ADR condition for $\partial\oT$.
Let $x \in \partial\oT$ and $0 < r \le \diam(Q_0)$. Since 
$\partial\oT \subseteq \cup_{j=0}^\infty \partial \oT_j$,
\begin{equation}\label{subadd}
H^n(B(x,r) \cap \partial\oT) \le \sum_{j=0}^\infty H^n(B(x,r) \cap \partial\oT_j)
\end{equation}
If $B(x,r)$ meets $\partial\oT_0$, then there is an $x_0 \in B(x,r) \cap \partial\oT_0$, 
and $B(x_0,2r) \supset B(x,r)$.  By the ADR property for $\oT_0$
\begin{equation}\label{ADR eq1}
H^n(B(x,r) \cap \partial\oT_0) \le H^n(B(x_0, 2r) \cap \partial\oT_0) \lesssim r^n\,.
\end{equation}
 Next, we consider the contributions of $ \partial\oT_j$, $j\geq 1$, which we write as
$$\sum_{j=1}^\infty H^n(B(x,r) \cap \partial\oT_j) = \sum_{j: \, \ell(Q_j')\,>\,r}\,+\,  
\sum_{j: \, \ell(Q_j')\,\leq\, r}=: I +II\,.$$
We recall that by construction $\partial\oT_j\subset \overline{V(Q_j')}$, so by
Remark \ref{remark4.18}, the boundaries $\partial\oT_j$ are pairwise disjoint.
Thus, only a bounded number of terms can appear in the sum $I$,
so the desired bound $I\lesssim r^n$ follows by the ADR property of 
each $\partial\oT_j$.
Moreover, for each $j$, the diameter of $\oT_j$ is comparable to $\ell(Q_j')$,
and therefore 
the cubes $Q_j'$ appearing in
$II$ are all contained in $B(x,Cr)$, for some sufficiently large constant $C$ depending only on 
allowable parameters.  Consequently, by the ADR property of $\partial\oT_j$ and of $E$,
$$II \lesssim \sum_{j:\, Q_j'\subset B(x,Cr)} H^n(\partial\oT_j)\,\approx \!
\sum_{j:\, Q_j'\subset B(x,Cr)} \sigma(Q_j')\, \leq \,\sigma(E\cap B(x,Cr)) \,\approx r^n\,.$$  
Thus, we obtain the upper ADR bound. 

With the upper ADR bound in hand we now know that $\partial\oT$ is a set of locally finite perimeter, 
so by the isoperimetric inequality \cite[p. 222]{EG}, 
and the strong two sided Corkscrew condition for open sets, the lower ADR bound follows.
The last and easiest thing to show is that condition \ref{big pieces of omega} 
holds for $\oT$,
but this follows readily from \ref{little domains are a big portion}, and the fact that
by the induction hypothesis, \ref{big pieces of omega} 
hold for each $Q'_j$.
\end{proof}

\section{Proof of Theorem \ref{t2}}

The proof of Theorem $\ref{t2}$ will be immediate from the following Lemma

\begin{lemma}\label{t2lemma} Let $\om$ be a bounded open set in $\ree$, 
with $n$-dimensional ADR boundary, such that 
$\om := \cup_j \om_j$ is the union of its connected components $\om_j$.
Suppose further that each component is a chord arc domain with uniform 
bounds on the chord arc constants. Then for $x_0 \in \partial\om$ and 
$r_0 < \tfrac{1}{2}\diam{\om}$, and for $Y \in \om \setminus B(x_0,2r_0)$, let $\hm^Y$ denote 
the harmonic measure associated to $\om$, and set $\Delta_0 := B(x_0,r_0) \cap \pom$. Then 
$\hm^Y \in $ weak-$A_\infty(\Delta_0)$, with uniform control on the weak-$A_\infty$ constants.
\end{lemma}
\begin{proof}  Fix $B_0:= B(x_0,r_0)$, with $x_0\in \pom$, and $r_0 < \diam(\pom)$.
Set $\Delta = \Delta(x,r)$ and $2\Delta = \Delta(x,2r)$, with $x\in\pom$, and suppose that 
$B(x,2r)\subset B_0$ (thus, $r_0\geq 2r$).   Recalling the definition of weak-$A_\infty$,
we need to show there exist uniform positive constants $C$ and $\theta$ such that for each Borel set
$A \subseteq \Delta$ 
\begin{equation}\label{aihmeq1}
\hm^Y(A) \le C \left(\frac{\sigma(A)}{\sigma(\Delta)}\right)^\theta \hm(2\Delta)\,,
\end{equation}
whenever $Y\in \Omega \setminus B(x_0,2r_0)$.  Let us fix such a point $Y$.
We note that $Y \in \om_j$ for some $j$ and therefore $\hm^Y$ is just the 
harmonic measure associated to the domain $\om_j$.   Thus, if $A \cap \pom_j = \emptyset$, 
then \ref{aihmeq1} holds trivially.   We may therefore 
assume that this is not the case. Let $z \in A \cap \pom_j$, set $A' = A \cap \pom_j$, and set $\sigma_j = H^n |_{\pom_j}$. Notice that $\dist(z,Y) \ge r_0 \geq 2r$, since in particular,
$z\in  B_0$, while $Y\in \om\setminus 2B_0$.  Thus,
the diameter of $\om_j$ must be greater than $r$.
Moreover, by the result obtained independently in \cite{DJe} and in \cite{Se}, 
$\hm^Y \in A_\infty(\Delta_\star)$, where
$\Delta_\star = B(z,r) \cap \om_j$. 
Note that $\Delta_\star \subset 2\Delta$, and also that 
by the uniform ADR property, $\sigma(\Delta) \approx \sigma_j(\Delta_\star) = \sigma(\Delta_\star)$. 

Since $\hm^Y \in A_\infty(\Delta_\star)$, we have that 
\begin{align} \nonumber
\hm^Y(A) &=\hm^Y(A') \\ \nonumber
&\leq C \left(\frac{\sigma_j(A')}{\sigma_j(\Delta_\star)}\right)^\theta\hm^Y(\Delta_\star)\\ \nonumber
& =C \left(\frac{\sigma(A')}{\sigma(\Delta_\star)}\right)^\theta\hm^Y(\Delta_\star) 
\, \lesssim  \left(\frac{\sigma(A)}{\sigma(\Delta)}\right)^\theta\hm^Y(2\Delta).
\end{align}
\end{proof}

To prove Theorem \ref{t2}, we apply the preceding
lemma to each domain $\widetilde{\om}$ constructed in Theorem \ref{t1}. 
We need only verify that the point $Y\in \widetilde{\om} \setminus 2B_0$ may be replaced by
any Corkscrew point $X_{\Delta_0}\in \widetilde{\om}$, relative to the surface ball
$\Delta_0$.  To this end, we fix such a Corkscrew point $X_{\Delta_0}$,
and observe that $X_{\Delta_0}\in \widetilde{\om}\setminus 2B_1$, 
for any ball $B_1$ meeting $B_0$, centered on $\partial\widetilde{\om}$,
with radius $r_1= \eta r_0$, if $\eta>0$ is chosen small enough
depending only on the Corkscrew constant for $\widetilde{\om}$.
Then for each such $B_1$, by Lemma \ref{t2lemma},
$\hm_0:=\hm^{X_{\Delta_0}}$ belongs to weak-$A_\infty(\Delta_1)$,
where $\Delta_1= B_1\cap \partial\widetilde{\om}$.  Now suppose that $B(x,2r)\subset B_0$,
with $x\in  \partial\widetilde{\om}$, and observe that $B:=B(x,r)$ may be covered by 
a collection of balls $\mathcal{F}=\{B^i: =B(x^i,r^i)\}$, of bounded cardinality depending only on dimension and $\eta$,
such that $r^i\approx r$, $2B^i\subset 2B$, and each $B^i$ is contained in some ball $B_1=B_1^i$ 
as above.  Set $\Delta = B\cap \partial\widetilde{\om}$, $\Delta^i := B^i\cap \partial\widetilde{\om}$,
and for $A\subset \Delta$, let $A_i:= A\cap \Delta^i$.   We then apply \eqref{aihmeq1} to each $A_i, \Delta^i$,
and use that for each $i$,
$H^n(\Delta) \approx H^n(\Delta^i)$  (depending on $\eta$ and the ADR constants).
Since $\#\mathcal{F}$ is bounded, we may then sum in $i$ to obtain Theorem \ref{t2}.

We conclude by observing that Lemma \ref{t2lemma}, and hence Theorem \ref{t1},
also apply to the Riesz measure (``$p$-harmonic measure")
associated to the $p$-Laplace equation $$\Delta_p u := \dv\big(|\nabla u|^{p-2} \nabla u\big) =0\,.$$
\begin{definition}\label{phm1}($p$-harmonic measure). Let $\om \subset \ree$ be open. For $x \in \pom$, $0 < r < \tfrac{1}{8}\diam(\om)$,
suppose that $u \ge 0$, 
with $u \equiv 0$ on $\pom$, $\Delta_p u = 0$ in $4B \cap \om$. We define the $p$-harmonic 
measure $\mu$ associated to $u$, as the unique finite positive Borel measure such that
\begin{equation}\label{phm2}
-\iint_\om |\nabla u|^{p-2} \nabla u \cdot \nabla \Phi \, dx = \int_{\pom} \Phi \, d\mu \quad \forall \Phi \in C_0^\infty(4B).
\end{equation}
\end{definition}

\begin{proposition} Let $\om$ be a bounded open set in $\ree$, 
with $n$-dimensional ADR boundary, such that 
$\om := \cup_j \om_j$ is the union of its connected components $\om_j$, and
satisfies a strong 2-sided corkscrew condition. 
Suppose further that each component $\om_j$
is a chord arc domain with uniform bounds on the chord arc constants. 
Then for $x \in \partial\om$ and $r < \tfrac{1}{8}\diam(\om)$, let $u,\,\mu$ 
be as above, set $\hat{\mu} :=  \mu|_{\partial\om_j}$ and let $\hat{c}$ be the corkscrew constant for the set $\om$. 
If $\diam(\om_j) > 2\hat{c}r$, then $\hat{\mu} \in$ weak-$A_\infty(\Delta)$. 
In particular, given $x \in \pom$ and $r \in (0,\tfrac{1}{8}\diam(\om))$, if we
let $\om_\Delta$ be some component which contains an interior corkscrew point relative to the ball $B(x,r)$, then
$\mu|_{\pom_\Delta} \in$ weak-$A_\infty(\Delta)$.
\end{proposition}
\begin{proof}
The proof is the same as that of Lemma \ref{t2lemma}. Lewis and Nystr\"om \cite{LN} showed 
that for an NTA domain $D$ with ADR boundary, if $\mu$ is $p$-harmonic measure associated to a non-trivial
positive $p$-harmonic function $u$ in $B(x,4r)\cap D$, with $x\in \partial D$, $r< \diam(D) /4$,
then $\mu \in A_\infty(\partial D\cap B(x,r))$, with respect to
$\sigma_D:= H^n|_{\partial D}$; we use this fact in place of the result of \cite{DJe}, \cite{Se}. 
\end{proof}


\begin{thebibliography}{AHMTT1}
\parskip=0.1cm



\bibitem[AHLT]{AHLT} P. Auscher, S. Hofmann, J.L. Lewis and P. Tchamitchian,  Extrapolation of Carleson measures and the analyticity of Kato's square-root operators,  {\em Acta Math.}     \textbf{187}  (2001),  no. 2, 161--190.

\bibitem[AHMTT]{AHMTT} P. Auscher, S. Hofmann, C. Muscalu, T. Tao and C. Thiele,
Carleson measures, trees, extrapolation, and $T(b)$ theorems,  {\em Publ. Mat.} {\bf 46}
(2002),  no. 2, 257--325.

\bibitem[AHMNT]{AHMNT} J. Azzam, S. Hofmann,  J. M.  Martell, K. Nystr{\"o}m, T. Toro, 
A new characterization of chord-arc domains, preprint 2014.  arXiv:1406.2743

\bibitem[AS]{AS} J.Azzam and Raanan Schul, Hard Sard: quantitative implicit function and extension theorem for Lipschitz maps,
{\it Geom. Funct. Anal.} {\bf 22} (2012), 1062-1123

\bibitem[BiJ]{BiJo} C. Bishop and P. Jones, Harmonic measure and arclength,
{\it Ann. of Math. (2)} {\bf 132} (1990), 511--547.

\bibitem[BL]{BL} B. Bennewitz and J.L. Lewis, On weak reverse H\"{o}lder inequalities
for nondoubling harmonic measures, {\it Complex Var. Theory Appl.} {\bf 49} (2004), 571-582.




\bibitem[Car]{Car} L. Carleson,  Interpolation by bounded analytic functions and the corona problem,
{\it Ann. of Math. (2)} {\bf 76} (1962), 547--559.

\bibitem[CG]{CG} L. Carleson and J. Garnett, Interpolating sequences and separation properties,
{\it J. Analyse Math.} {\bf 28} (1975), 273--299.


\bibitem[Chr]{Ch} M. Christ,  A $T(b)$ theorem with remarks on analytic
capacity and the Cauchy integral, {\it Colloq. Math.}, {\bf LX/LXI} (1990), 601--628.

\bibitem[Da]{D1} B. Dahlberg, On estimates for harmonic measure,
{\it Arch. Rat. Mech. Analysis} {\bf 65} (1977), 272--288.





\bibitem[DJ]{DJe} G. David and D. Jerison, Lipschitz approximation
to hypersurfaces, harmonic measure,
and singular integrals, {\it Indiana Univ. Math. J.} {\bf 39} (1990),
no. 3, 831--845.

\bibitem[DS1]{DS1} G. David and S. Semmes,
Singular integrals and rectifiable sets in $\re^n$: Beyond Lipschitz graphs, {\it Asterisque} {\bf 193} (1991).

\bibitem[DS2]{DS2} G. David and S. Semmes, {\it Analysis of and on
Uniformly
Rectifiable Sets}, Mathematical Monographs and Surveys {\bf 38}, AMS
1993.

\bibitem[DS3]{DS3} G. David and S. Semmes, Quantitative Rectifiability and Lipschitz Mappings, {\it Trans. Amer. Math. Soc.} {\bf 337} (1993), no. 2, 855-889

\bibitem[EG]{EG} L.C. Evans and R.F. Gariepy, {\it Measure Theory and Fine
Properties of Functions}, Studies in Advanced Mathematics, CRC Press,
Boca Raton, FL, 1992.




\bibitem[HL]{HL}  S. Hofmann and J.L. Lewis, The Dirichlet problem for parabolic operators
with singular drift terms, {\it Mem. Amer. Math. Soc.} \textbf{151} (2001), no. 719.

\bibitem[HM1]{HM-TAMS} S. Hofmann and J.M. Martell, $A_\infty$ estimates via extrapolation of Carleson measures
and applications to divergence form elliptic operators, {\it Trans. Amer. Math. Soc.} \textbf{364} (2012), no. 1, 65--101

\bibitem[HM2]{HM-I} S. Hofmann and J.M. Martell, Uniform rectifiability and harmonic measure I: Uniform rectifiability implies Poisson kernels in $L^p$, {\it Ann. Sci. \'Ecole Norm. Sup.}, to appear.

\bibitem[HM3]{HM-IV} S. Hofmann and J.M. Martell,
Uniform Rectifiability
and harmonic measure IV:  Ahlfors regularity plus Poisson kernels in $L^p$
implies uniform rectifiability.


\bibitem[HMM]{HMM2} S. Hofmann, J.M. Martell, and S. Mayboroda,  Uniform rectifiability, Carleson measure estimates and approximation of harmonic functions, preprint  


\bibitem[HMU]{HMU} S. Hofmann, J.M. Martell and I. Uriarte-Tuero, Uniform rectifiability
and harmonic measure II:  Poisson kernels in $L^p$ imply uniform rectifiability, 
{\it Duke Math. J.} \textbf{163} (2014), no. 8, 1601--1654.





\bibitem[HMMM]{HMMM} S. Hofmann, D.  Mitrea, M. Mitrea, A. Morris,
$L^p$-Square Function Estimates on Spaces of Homogeneous Type and on Uniformly Rectifiable Sets,
preprint 2103.  arXiv:1301.4943



\bibitem[JK]{JK} D. Jerison and C. Kenig,  Boundary behavior of
harmonic functions in nontangentially accessible domains, {\it Adv. in Math.}
\textbf{46} (1982), no. 1, 80--147.



\bibitem[La]{Lav} M. Lavrentiev, Boundary problems in the theory of univalent functions (Russian),
{\it Math Sb.} {\bf 43} (1936), 815-846;  {\it AMS Transl. Series 2} {\bf 32} (1963), 1--35.

\bibitem[LM]{LM} J. Lewis and M. Murray,  The method of layer potentials for the heat equation in time-varying domains, {\it Mem. Amer. Math. Soc.} \textbf{114} (1995), no. 545.

\bibitem[LN]{LN} J. L. Lewis and K. Nystr\"om, 
Regularity and free boundary regularity for the $p$-Laplace operator in Reifenberg flat and Ahlfors regular domains, {\it Journal Amer. Math. Soc.} {\bf 25} (2012), 827-862.

\bibitem[MMV]{MMV} P. Mattila, M. Melnikov and J.\,Verdera, The Cauchy integral, analytic capacity, and uniform rectifiability,
{\it Ann. of Math. (2)} \textbf{144} (1996), no. 1, 127--136.


\bibitem[NToV]{NToV} F. Nazarov, X. Tolsa, and A. Volberg, On the uniform rectifiability of ad-regular measures with bounded Riesz
transform operator: The case of codimension 1, {\it Acta Math.}, to appear.

\bibitem[RR]{Rfm} F. and M. Riesz, \"Uber die randwerte einer analtischen funktion,
{\it Compte Rendues du Quatri\`eme Congr\`es des Math\'ematiciens Scandinaves}, Stockholm 1916,
Almqvists and Wilksels, Upsala, 1920.

\bibitem[Se]{Se} S. Semmes, Analysis vs. geometry on a class of rectifiable hypersurfaces in
$\mathbb{R}^n$, {\it Indiana Univ. Math. J.} {\bf 39} (1990), 1005--1035.

\bibitem[Ste]{St} E. M. Stein, {\it Singular Integrals and Differentiability Properties
of Functions}, Princteon University Press, Princeton, NJ, 1970.


\end{thebibliography}
\end{document}